\documentclass[11pt]{article}

\usepackage{amsmath}
\usepackage{amsthm}
\usepackage{amssymb}
\usepackage{amsfonts}
\usepackage{mathrsfs}
\usepackage{enumerate}
\usepackage{url}

\usepackage{mathptmx}

\usepackage{color}

\newcommand{\change}[1]{#1}

\newtheorem{theorem}{Theorem}
\newtheorem{proposition}[theorem]{Proposition}
\newtheorem{lemma}[theorem]{Lemma}
\newtheorem{corollary}[theorem]{Corollary}

\newtheorem*{theorem*}{Theorem}
\newtheorem*{theorem1}{\change{Theorem \ref{thm:main}}}
\newtheorem*{theorem2}{Theorem \ref{thm:main2} (restated)}
\newtheorem*{theorem3}{\change{Theorem \ref{thm:negwkl}}}
\newtheorem*{theorem4}{\change{Theorem \ref{thm:from-tanaka}}}
\newtheorem*{Tanaka}{\change{Tanaka's self-embedding theorem for $\wklstar$}}

\theoremstyle{definition}
\newtheorem{definition}[theorem]{Definition}

\theoremstyle{remark}

\newtheorem*{ack}{Acknowledgement}
\newtheorem*{remark}{Remark}

\newtheorem{question}{Question}

\newcommand{\tuple}[1]{\langle #1 \rangle}

\newcommand{\bbn}{\mathbb{N}}

\newcommand{\bba}{\mathbb{A}}

\newcommand{\supexp}{\mathrm{superexp}}
\newcommand{\Log}{\mathrm{Log}}

\newcommand{\rca}{\mathsf{RCA}_0}
\newcommand{\rcastar}{\mathsf{RCA}^*_0}
\newcommand{\wkl}{\mathsf{WKL}_0}
\newcommand{\wklstar}{\mathsf{WKL}^*_0}
\newcommand{\wklrm}{\mathrm{WKL}}

\newcommand{\m}{{\mathcal M}}
\newcommand{\X}{{\mathcal X}}
\newcommand{\Y}{{\mathcal Y}}
\renewcommand{\P}{{\mathcal P}}

\DeclareMathOperator\lh{lh}

\newcommand{\as}[2]{\forall #1 \! < \! #2 \,}

\def\tuple#1{\langle#1\rangle}

\title{Categorical characterizations of the natural numbers require primitive recursion}
\author{Leszek Aleksander Ko{\l}odziejczyk\footnote{Institute of Mathematics, 
University of Warsaw, Banacha 2, 02-097 Warszawa, Poland, \texttt{lak@mimuw.edu.pl}.
Supported in part
by Polish National Science Centre grant no.~2013/09/B/ST1/04390.} 
\and
Keita Yokoyama\footnote{School of Information Science, Japan Advanced Institute of Science and
Technology, Nomi, Ishikawa, Japan,
\texttt{y-keita@jaist.ac.jp}.
Supported in part
by JSPS Grant-in-Aid for Research Activity Start-up grant no.~25887026.}
}

\begin{document}

\maketitle
\bibliographystyle{amsalpha}

\begin{abstract}
Simpson and Yokoyama [Ann.\ Pure Appl.\ Logic 164 \change{(2013)},  284--293] asked whether there exists a characterization
of the natural numbers by a second-order sentence which is provably categorical in the theory $\rcastar$. We answer in the negative, showing that for any characterization of the natural numbers which is provably true in $\wklstar$, the categoricity theorem implies $\Sigma^0_1$ induction. 

On the other hand, we show that $\rcastar$ does make it possible to characterize the natural
numbers categorically by means of a set of second-order sentences. We also show that a certain $\Pi^1_2$-conservative
extension of $\rcastar$ admits a provably categorical single-sentence characterization of the naturals, but each such 
characterization has to be inconsistent with $\wklstar + \supexp$.  
\end{abstract}

Inspired by a question of V\"a\"an\"anen (see e.g.\ \cite{vaananen:sol} for some related work), Simpson and the second author \change{\cite{sy:peanocat}} studied various second-order characterizations of $\tuple{\bbn, S, 0}$,
with the aim of determining the reverse-mathematical strength of 
their respective categoricity theorems. One of the general conclusions
is that the strength of a categoricity theorem depends
heavily on the characterization. Strikingly, however, each of the categoricity theorems considered in \change{\cite{sy:peanocat}}
 implies $\rca$, even over the much weaker base theory $\rcastar$, that is, $\rca$ with $\Sigma^0_1$ induction 
replaced by $\Delta^0_0$ induction in the language with exponentiation.
{(For $\rcastar$, see \cite{simpson-smith}.)}

This leads to the following question.
\change{
\begin{question}\cite[Question 5.3, slightly rephrased]{sy:peanocat}\label{q:sy}
Does $\rcastar$ prove the existence of a second-order sentence or set of sentences $T$ such that $\tuple{\bbn, S, 0}$
is a model of $T$ and all models of $T$ are isomorphic to $\tuple{\bbn, S, 0}$? One may also
consider the same question with $\rcastar$ replaced by $\Pi^0_2$-conservative extensions of $\rcastar$.
\end{question}
}

\change{Naturally, to have any hope of characterizing infinite structures categorically, second-order logic 
has to be interpreted according to the \emph{standard} 
semantics (sometimes also known as {strong} or {Tarskian} semantics), as opposed to the \emph{general} (or {Henkin}) semantics. In other words, a second-order quantifier $\forall X$ really means  
``for \emph{all} subsets of the universe''
(or, as we would say in a set-theoretic context, ``for all elements of the power set of the universe'').} 

\change{Question \ref{q:sy}} 
admits multiple versions depending on whether we focus on $\rcastar$ 
or consider other $\Pi^0_2$-equivalent theories and whether we want the characterizations 
of the natural numbers to be sentences or sets of sentences. 
The most basic version, restricted to $\rcastar$ and single-sentence characterizations, 
would read as follows:

\begin{question}\label{q:basic}
Does there exist a second-order sentence $\psi$ in the language with one unary function $f$ and one
constant $c$ such that $\rcastar$ proves: (i) $\tuple{\bbn, S, 0}~\models~\psi$, and (ii) for every $\tuple{A, f, c}$, 
if $\tuple{A, f, c} \models \psi$, then there exists an isomorphism between $\tuple{\bbn, S, 0}$ and $\tuple{A, f, c}$?
\end{question}

We answer Question \ref{q:basic}  in the negative. In fact, characterizing $\tuple{\bbn, S, 0}$ not only up to isomorphism, but 
even just up to \emph{equicardinality of the universe}, requires the full strength of $\rca$. More precisely:

\begin{theorem}\label{thm:main} Let $\psi$ be a second-order sentence in the language with one unary function $f$ 
and one individual constant $c$. If $\wklstar$ proves that $\tuple{\bbn, S, 0} \models \psi$, then over $\rcastar$ the statement 
``for every $\tuple{A, f, c}$, if $\tuple{A, f, c} \models \psi$, then there exists a bijection between 
$\bbn$ and $A$'' implies $\rca$.
\end{theorem}

Since $\rca$ is equivalent over $\rcastar$ to a statement expressing the correctness of defining functions by primitive recursion
{\cite[Lemma~2.5]{simpson-smith}}, Theorem \ref{thm:main} may be intuitively understood as saying that, for provably true 
single-sentence characterizations at least, ``categorical characterizations of the natural numbers require primitive recursion''. 

Do less stringent versions of Question \ref{q:sy} give rise to ``exceptions'' to this general conclusion? As it turns out, they do. Firstly, characterizing the natural numbers by a \emph{set} of sentences is already possible in $\rcastar$, \change{in the following sense (for a precise statement of the theorem, see Section 4)}:

\begin{theorem}\label{thm:main2}
There exists a $\Delta_0$-definable (and polynomial-time recognizable)
set $\Xi$ of $\Sigma^1_1 \wedge \Pi^1_1$ sentences such that $\rcastar$ proves: for every $\tuple{A, f, c}$, 
$\tuple{A, f, c}$ satisfies all $\xi \in \Xi$ if and only if \change{$\tuple{A, f, c}$} is isomorphic to $\tuple{\bbn, S, 0}$.
\end{theorem}

Secondly, even a single-sentence characterization is possible in a $\Pi^1_2$-conservative extension of $\rcastar$, at least if one is willing to consider rather peculiar theories:

\begin{theorem}\label{thm:negwkl}
There is a $\Sigma^1_2$ sentence which is a categorical characterization of $\tuple{\bbn, S, 0}$ provably in $\rcastar + \neg \wklrm$.
\end{theorem}

Theorem \ref{thm:negwkl} is not quite satisfactory, as the theory and characterization it speaks of are false in $\tuple{\omega, \P(\omega)}$. So, another natural question to ask is whether a single-sentence characterization of the natural numbers can be provably categorical in a \emph{true} $\Pi^0_2$-conservative extension of $\rcastar$. We show that under an assumption just a little stronger than $\Pi^0_2$-conservativity, the characterization from Theorem \ref{thm:negwkl} is actually ``as true as possible'':

\begin{theorem}\label{thm:from-tanaka}
Let $T$ be an  extension of $\rcastar$ conservative for first-order $\forall \Delta_0(\Sigma_1)$ sentences. Let $\eta$ be a second-order sentence consistent with $\wklstar + \supexp$. Then it is not the case that $\eta$ is a categorical characterization of $\tuple{\bbn, S, 0}$ provably in $T$.
\end{theorem}

The proofs of our theorems make use of a weaker notion of isomorphism 
to $\tuple{\bbn, S, 0}$ studied in \change{\cite{sy:peanocat}}, 
that of ``almost isomorphism''. Intuitively speaking, a structure $\tuple{A, f, c}$
satisfying some basic axioms is almost isomorphic to $\tuple{\bbn, S, 0}$ if it is ``equal to or shorter than''
the natural numbers. The two crucial facts we prove and exploit are that almost isomorphism to $\tuple{\bbn, S, 0}$
can be characterized by a single sentence provably in $\rcastar$, and that structures 
almost isomorphic to $\tuple{\bbn, S, 0}$ correspond to $\Sigma^0_1$-definable cuts.

The paper is structured as follows. After \change{a 
preliminary} Section 1, we conduct our study of almost isomorphism to $\tuple{\bbn, S, 0}$ in Section 2. We then prove Theorem \ref{thm:main} in Section 3, Theorems \ref{thm:main2} and \ref{thm:negwkl} in Section 4, and Theorem \ref{thm:from-tanaka} in Section~5.

\section{Preliminaries}
We assume familiarity with subtheories of second-order arithmetic, as presented in \cite{sosoa}. Of the \change{``Big Five''} theories
featuring prominently in that book, we only need the two weakest: $\rca$, axiomatized by $\Delta^0_1$ comprehension and $\Sigma^0_1$ induction (and a finite list of simple basic axioms), and $\wkl$, which extends $\rca$ by the axiom WKL stating that an infinite binary tree has an infinite branch.

We also make use of some well-known fragments of first-order arithmetic, principally 
$\mathrm{I}\Delta_0 + \exp$, which extends induction for $\Delta_0$ formulas by an axiom $\exp$ stating the totality of exponentiation; $\mathrm{\mathrm{B}\Sigma_1}$, which extends $\mathrm{I}\Delta_0$ by the $\Sigma_1$ collection (bounding) principle; and $\mathrm{I}\Sigma_1$. For a comprehensive treatment of these and other subtheories of first-order arithmetic, refer to \cite{hp93}. 

\change{The well-known hierarchies defined in terms of alternations of first-order quantifiers make sense both for purely first-order formulas and for formulas allowing second-order parameters, and we will need notation to distinguish between the two cases. For classes of formulas with first-order quantification but also arbitrary second-order parameters, we use the $\Sigma^0_n$ notation standard in second-order arithmetic. On the other hand, when discussing classes of first-order formulas, we adopt a convention often used in first-order arithmetic and omit the superscript ``$^0$''. Thus, for instance, a $\Sigma_1$ formula is a first-order formula (with no second-order variables at all) containing a single block of existential quantifiers followed by a bounded part. More generally, if we want to speak of a formula possibly containing second-order parameters $\bar X$ but no other second-order parameters, we use notation of the form $\Sigma_n(\bar X)$ (to be understood as ``$\Sigma_n$ relativized to $\bar X$'').}

A formula is $\Delta_0(\Sigma_1)$ if it belongs to the closure of $\Sigma_1$ under boolean operations and bounded first-order quantifiers. \change{$\forall \Delta_0(\Sigma_1)$ (respectively $\exists \Delta_0(\Sigma_1)$) is the class of first-order formulas
which consist of a block of universal (respectively existential) quantifiers followed by a $\Delta_0(\Sigma_1)$ formula.}

The theory $\rcastar$ was introduced in \cite{simpson-smith}. It differs from $\rca$ in that the $\Sigma^0_1$ induction axiom is replaced by $\mathrm{I}\Delta^0_0 + \exp$. $\wklstar$ is $\rcastar$ plus the WKL axiom. Both $\rcastar$ and $\wklstar$ have
$\mathrm{\mathrm{B}\Sigma_1} + \exp$ as their first-order part, while the first-order part of $\rca$ and $\wkl$ is $\mathrm{I}\Sigma_1$.

We let $\supexp$ denote both the ``tower of exponents'' function defined by
$\supexp(x) = \exp_x(2)$ (where $\exp_0(2) =1, \exp_{x+1}(2)= 2^{\exp_x(2)}$) 
and the axiom saying that for every $x$, $\supexp(x)$ exists. $\Delta_0(\exp)$ stands for the class of bounded formulas
in the language extending the language of Peano Arithmetic by a symbol for $x^y$. 
$\mathrm{I}\Delta_0(\exp)$ is a definitional extension of $\mathrm{I}\Delta_0 + \exp$.

In any model $M$ of a first-order arithmetic theory (possibly the first-order part of a second-order structure), a \emph{cut} is a nonempty subset of $M$ which is downwards closed and closed under successor. For a cut $J$, we sometimes abuse notation and also write $J$ to denote the structure $\tuple{J, S, 0}$, or even $\tuple{J, +, \cdot, \le , 0, 1 }$ if $J$ happens to be closed under multiplication. 

If $\tuple{M, \X} \models\rcastar$ and $J$ is a cut in $M$, then $\X_J$ will denote the family of sets $\{X \cap J: X \in \X\}$.
\change{Throughout the paper, we frequently use the following simple but important result without further mention.
\begin{theorem*}[\cite{simpson-smith}, Theorem 4.8]
If $\tuple{M, \X} \models\rcastar$ and $J$ is a proper cut in $M$ which is closed under $\exp$, then $\tuple{ J, \X_J} \models \wklstar$.
\end{theorem*}
}

\change{If $\tuple{M, \X} \models \rcastar$ and $A \in \X$, then $A$ is \emph{$M$-finite} (or simply \emph{finite} if we do not want to emphasize $M$) if there exists $a \in M$ such that all elements of $A$ are smaller than $a$. Otherwise, the set $A$ is \emph{($M$)-infinite}. For each $M$-finite set $A$ there is an element $a \in M$ coding $A$ in the sense that $A$ consists exactly of those $x \in M$ for which the $x$-th bit in the binary notation for $a$ is $1$. Moreover, $\rcastar$ has a well-behaved notion of cardinality of finite sets, which lets us define the \emph{internal cardinality}  $|A|_\m$ of any $A \in \X$ as $\sup(\{x \in M: A \textrm{ contains a finite subset with at least } {x} \textrm{ elements}\})$. $|A|_\m$ is an element of $M$ if $A$ is $M$-finite, and a cut in $M$ otherwise.}

$\bbn$ stands for the set of numbers defined by the formula $x = x$; in other words, $\bbn_M = M$. To refer to the set of standard natural numbers, we use the symbol $\omega$. \change{The general notational conventions regarding cuts apply also to $\bbn$: for instance, if there is no danger of confusion, we sometimes write that some structure is ``isomorphic to $\bbn$'' rather than ``isomorphic to $\tuple{\bbn, S,0}$''.}

\change{We will be interested mostly in structures of the form $\tuple{A, f, c}$, where $f$ is a unary function and $c$ an individual constant. The letter $\bba$ will always stand for some structure of this form. $\bba$ is a \emph{Peano system} if $f$ is one-to-one, $c \notin \mathrm{rng}(f)$, and $\bba$ satisfies the second-order induction axiom:
\begin{equation}\label{eqn:induction}\forall X \left[X(c) \wedge \forall a\, [X(a) \rightarrow X(f(a))] \rightarrow \forall a\, X(a)\right].\end{equation}
Second-order logic is considered here in its full version ---  that is, non-unary second-order quantifiers are allowed --- and interpreted according to the so-called standard semantics (cf.\ e.g.\ \cite{sep:sol}). Thus, the quantifier $\forall X$ with $X$ unary means ``for \emph{all} subsets of $A$'', $\forall X$ with $X$ binary means ``for \emph{all} binary relations on $A$'', etc. 
For instance, $\bba$ satisfies (\ref{eqn:induction}) exactly if there is no proper subset 
of $A$ containing $c$ and closed under $f$. Of course, from the perspective 
of a model $\m = \tuple{M, \X}$ of $\rcastar$ or some other 
fragment of second-order arithmetic, ``for all subsets of $A$'' means
``for all $X\in \X$ such that $X\subseteq A$''. After all, 
according to $\m$ there are no other subsets of $A$!}

\section{Almost isomorphism}\label{section:almost-iso}

A Peano system is said to be \emph{almost isomorphic to $\tuple{\bbn, S, 0}$}  if for every $a \in A$ there is some ${x \in \bbn}$ such that $f^x(c) = a$. \change{Here we take $f^x(c) = a$ to mean that there exists a sequence $\tuple{a_0, a_1, a_2, \ldots, a_x}$ such that $a_0 = c$, $a_{z+1} = f(a_z)$ for $z< x$, and $a_x = a$. Note that we need to explicitly assert the existence of this sequence, which we often refer to as $\tuple{c, f(c), f^2(c), \ldots, f^x(c)}$, because $\rcastar$ is too weak to prove 
that any function can be iterated an arbitrary number of times.}



Being almost isomorphic to $\bbn$ is a definable property:

\begin{lemma}\label{lem:almost-cat} 
There exists a $\Sigma^1_1 \wedge \Pi^1_1$ sentence $\xi$ in the language with one unary function $f$ and one individual constant $c$
such that $\rcastar$ proves: for every $\bba$, $\bba \models \xi$ if and only if $\bba$ is a Peano system
almost isomorphic to $\tuple{\bbn, S, 0}$.
\end{lemma}

\begin{proof}
By definition, $\bba$ is a Peano system precisely if it satisfies the $\Pi^1_1$ sentence $\xi_{\mathrm{peano}}$:
\[f \textrm{ is 1-1} \wedge c \notin \mathrm{rng}(f) \wedge \forall X \left[X(c) \wedge \forall a\, [X(a) \rightarrow X(f(a))] \rightarrow \forall a\, X(a)\right].\]
The sentence $\xi$ will be the conjunction of $\xi_{\mathrm{peano}}$, the $\Sigma^1_1$ sentence $\xi_{\preccurlyeq,\Sigma}$:
\begin{center}
there exists a discrete linear ordering $\preccurlyeq$ \\ for which $c$ is the least element and $f$ is the successor function,
\end{center}
and the $\Pi^1_1$ sentence $\xi_\change{{\preccurlyeq, \Pi}}$:
\begin{center}
for every \change{discrete} linear ordering $\preccurlyeq$ with $c$ as least element and $f$ as successor \\
and for every $a$, the set of elements $\preccurlyeq$-below $a$ is Dedekind-finite. 
\end{center}

We say that a set $X$ is {\emph{Dedekind-finite}} if there is no bijection between $X$ and a proper subset of $X$. \change{Note that $\xi$ involves quantification over non-unary relations: linear orderings and (graphs of) bijections.}

\change{In verifying that $\xi$ characterizes Peano systems almost isomorphic to $\bbn$, we will make use of the fact that provably in $\rcastar$,  for any set $A$ and any $X \subseteq A$, $A \models \textrm{``}X \textrm{ is Dedekind-finite''}$ exactly if $X$ is finite. To see that this is true, note that if $X$ is infinite, then the map which takes 
$x \in X$ to the smallest $y \in X$ such that $x < y$ is a bijection between $X$ and its proper subset $X\setminus \{\min X\}$, and the graph of this bijection is a binary relation on $A$ witnessing $A \not \models \textrm{``}X \textrm{ is Dedekind-finite''}$. On the other hand, any witness for $A \not \models \textrm{``}X \textrm{ is Dedekind-finite''}$ must in fact be the graph of a bijection between $X$ and a proper subset of $X$, but such a bijection cannot exist for finite $X$ because all proper subsets of a finite set have strictly smaller  cardinality than the set itself. }

We first prove that Peano systems almost isomorphic to $\bbn$ satisfy $\xi_{\preccurlyeq,\Sigma}$ and $\xi_{\preccurlyeq,\Pi}$. Let $\bba$
be almost isomorphic to $\bbn$. Every $a \in A$ is of the form $f^x(c)$ for some $x \in \bbn$. Moreover, $x$ is unique. To see this, assume that $a~=~f^x(c)~=~f^{x+y}(c)$ and that 
$\tuple{c, f(c), \ldots, f^x(c)=a , f^{x+1}(c), \ldots, f^{x+y}(c)=a}$
is the sequence witnessing that $f^{x+y}(c)=a$ (by $\Delta^0_0$-induction, this sequence is unique and its first $x+1$ elements comprise the unique sequence witnessing $f^{x}(c)=a$).
If $y > 0$, then we have $c \neq f^{y}(c)$ and then $\Delta^0_0$-induction coupled with the injectivity of $f$ gives 
$f^{w}(c) \neq f^{w+y}(c)$ for all $w \le x$. So, $y = 0$.

Because of the uniqueness of the $f^x(c)$ representation for $a \in A$, we can define $\preccurlyeq$ on $A$ by $\Delta^0_1$-comprehension in the following way:
\[a \preccurlyeq b := \exists x\, \exists y\, (a = f^x(c) \wedge b = f^{y}(c) \wedge x \le y).\]
Clearly, $\preccurlyeq$ is a discrete linear ordering on $A$ with $c$ as the least element and $f$ as the successor function, so 
$\bba$ satisfies $\xi_{\preccurlyeq,\Sigma}$.

For each $a \in A$, the set of elements $\preccurlyeq$-below $a$ is finite. Moreover, if $\lessdot$ is any ordering 
of $A$ with $c$ as least element and $f$ as successor, then for each $a \in A$ the set
\[ \{b \in A : b \preccurlyeq a \Leftrightarrow b \lessdot a \} \]
contains $c$ and is closed under $f$. Since $\bba$ is a Peano system, $\lessdot$ has to coincide with $\preccurlyeq$. Thus, $\bba$ satisfies $\xi_{\preccurlyeq,\Pi}$.

For a proof in the other direction, let  $\bba$ be a Peano system satisfying $\xi_{\preccurlyeq,\Sigma}$  and $\xi_{\preccurlyeq,\Pi}$. Let $\preccurlyeq$ be an ordering on $A$ witnessing $\xi_{\preccurlyeq,\Sigma}$. Take some $a \in A$. By $\xi_{\preccurlyeq,\Pi}$, the set $[c,a]_\preccurlyeq$ of elements $\preccurlyeq$-below $a$ is finite. Let $\ell$ be the cardinality of $[c,a]_\preccurlyeq$ and let $b$ be the $\le$-maximal element of $[c,a]_\preccurlyeq$. By $\Delta^0_0(\exp)$-induction on $x$ prove that there is an element below $b^{x+1}$ coding a  sequence $\tuple{s_0, \ldots, s_x}$ such that $s_0 = c$ and for all $y<x$, either $s_{y+1} = f(s_y) \preccurlyeq a$ or $s_{y+1} = s_y = a$. Take such a sequence  for $x = \ell-1$. If $a$ does not appear in the sequence, then by $\Delta^0_0(\exp)$-induction the sequence has the form $\tuple{c, f(c), \ldots, f^{\ell-1}(c)}$ and all its entries are distinct elements of $[c,a]_\preccurlyeq \setminus\{a\}$; an impossibility, given that $[c,a]_\preccurlyeq \setminus\{a\}$ only has $\ell -1$ elements. So, $a$ must appear somewhere in the sequence. Taking $w$ to be the least such that $a = s_w$, we easily verify that $a = f^w(c)$. 
\end{proof}

\begin{remark}
We do not know whether in $\rcastar$ it is possible to characterize $\tuple{\bbn, S, 0}$ up to almost isomorphism by a $\Pi^1_1$ sentence. This does become possible in the case of $\tuple{\bbn, \le}$ (given a suitable definition of almost isomorphism, cf.\ \change{\cite{sy:peanocat}}), where there is no need for the $\Sigma^1_1$ part of the characterization which guarantees the existence of a suitable ordering.
\end{remark}

An important fact about Peano systems almost isomorphic to $\bbn$ is that their isomorphism types correspond to $\Sigma^0_1$-definable cuts. This correspondence, which will play a major role in the proofs of our main theorems, is formalized in the following definition and lemma. 

\begin{definition} Let $\m = \tuple{M, \X}$ be a model of $\rcastar$. For a Peano system $\bba$ in $\m$ which is almost isomorphic to $\tuple{\bbn, S, 0}$, let $J(\bba)$ be the cut defined in $\m$ by the $\Sigma^0_1$ formula $\varphi(x)$:
\[\exists a \in A\, f^x(c) = a.\]
For a $\Sigma^0_1$-definable cut $J$ in $\m$, let the structure $\bba(J)$ be $\tuple{A_J, f_J, c_J}$, where the set $A_J$ consists of all the pairs $\tuple{x,y_x}$ such that $y_x$ is the smallest witness for the formula $x \in J$, the function $f_J$ maps $\tuple{x,y_x}$ to $\tuple{x+1, y_{x+1}}$, and $c_J$ equals $\tuple{0, y_0}$.
\end{definition}

\begin{lemma} \label{lem:system-cut}
Let $\m = \tuple{M, \X}$ be a model of $\rcastar$. The following holds:
\begin{itemize}
\item[(a)] for a $\Sigma^0_1$-definable cut $J$ in $\m$, the structure $\bba(J)$ is a Peano system almost isomorphic to $\tuple{\bbn, S, 0}$, and $J(\bba(J)) = J$,
\item[(b)] if $\bba \in \X$ is a Peano system almost isomorphic to $\tuple{\bbn, S, 0}$, then there is an isomorphism in $\m$ between $\bba(J(\bba))$ and $\bba$,
\item[(c)]  if $\bba \in \X$ is a Peano system almost isomorphic to $\tuple{\bbn, S, 0}$, then there is an isomorphism in $\m$ between $\bba$ and $J(\bba)$, which also induces an isomorphism between the second-order structures $\tuple{\bba, \X \cap \P(A)}$ and $\tuple{J(\bba), \X_{J(\bba)}}$.
\end{itemize}
\end{lemma}

Although all the isomorphisms between first-order structures mentioned in Lemma \ref{lem:system-cut} are elements of $\X$, a cut is not itself an element of $\X$ unless it equals $M$ (because induction fails for the formula $x \in J$ whenever $J$ is a proper cut). Obviously, the isomorphism between second-order structures mentioned in part (c) is also outside $\X$.

\begin{proof}
For a $\Sigma^0_1$-definable cut $J$ in $\m$, it is clear that $A_J$ and $f_J$ are elements of $\X$, that $f_J$ is an injection from $A_J$ into $A_J$, and that $c_J$ is outside the range of $f_J$. Furthermore, for every $\tuple{x, y_x} \in A_J$,
$\Sigma^0_1$ collection in $\m$ guarantees that there is a common upper bound on $y_0, \ldots, y_x$, so $\Delta^0_0$ induction is enough to show that the sequence $\tuple{c_J, f_J(c_J),\ldots, f^x_J(c_J) = \tuple{x, y_x}   }$ exists. If $X \subseteq A_J$, $X \in \X$, is such that $c_J \in X$ but $f^x_J(c_J) \notin X$, then $\Delta^0_0$ induction along the sequence $\tuple{c_J, f_J(c_J),\ldots, f^x_J(c_J)}$ finds some $w < x$ such that $f^w_J(c_J) \in X$ but $f_J(f^w_J(c_J))  \notin X$. Thus, 
$\bba(J)$ is a Peano system almost isomorphic to $\bbn$, and clearly $J(\bba(J))$ equals $J$, so part (a) is proved.

For part (b), if $\bba$ is almost isomorphic to $\bbn$, then each $a \in A$ has the form $a = f^x(c)$ for some $x \in J(\bba)$, and we know from the proof of Lemma \ref{lem:almost-cat} that the element $x$ is unique. Thus, the mapping which takes $f^x(c) \in \bba$ to $\tuple{x, y_x} \in \bba(J(\bba))$ is guaranteed to exist in $\m$ by $\Delta^0_1$ comprehension. It follows easily from the definitions of $J(\bba)$ and $\bba(J)$ that the mapping $f^x(c) \mapsto \tuple{x, y_x}$ is an isomorphism between $\bba$ and $\bba(J(\bba))$.

For part (c), we assume that $\bba$ equals $\bba(J(\bba))$, which we may do w.l.o.g.\ by part (b). The isomorphism between $\bba$ and $J(\bba)$ is given by $\tuple{x, y_x} \mapsto x$. To prove that this  also induces an isomorphism between $\tuple{\bba, \X \cap \P(A)}$ and $\tuple{J(\bba), \X_{J(\bba)}}$, 
we have to show that for any $X \subseteq A$, it holds that $X \in \X$ exactly if $\{x: \tuple{x,y_x} \in X\}$ has the form $Z \cap J(\bba)$ for some
$Z \in \X$. \change{This is easy if $J(\bba) = M$, so below we assume $J(\bba) \neq M$.}

The ``if'' direction is immediate: given $Z \in \X$, the set $\{\tuple{x,y_x}: x \in Z\}$ is $\Delta_0(Z)$ and thus belongs to $\X$.

To deal with the other direction, we assume that $\m$ is countable. We can do this w.l.o.g.\ because $J(\bba)$ is a definable cut, so the existence of a counterexample in some model would imply the existence of a counterexample in a countable model by a downwards Skolem-L\"owenheim argument.

By \cite[Theorem 4.6]{simpson-smith}, the countability of $\m$ means that we can extend $\X$ to a family $\X^+ \supseteq \X$ such that $\tuple{M, \X^+} \models \wklstar$. Note that there are no \change{$M$}-finite sets in $\X^+ \setminus \X$. 
\change{This is because for an $M$-finite set $X \in \X^+$ there is some $z \in M$ such that \[X = \{x: \textrm{ the }x\textrm{-th bit in the binary notation for }z\textrm{ is }1\}.\] 
Therefore, $X$ is $\Delta_0$-definable with parameter $z$ and so $X \in \X$.}

Now consider some $X \in \X$, $X \subseteq A$. Let $T$ be the set consisting of the finite binary strings $s$ satisfying:
\[\as{a,x}{\lh(s)} \big[ (a= \tuple{x,y_x} \wedge a \in X \rightarrow (s)_{x} =1) \wedge (a= \tuple{x,y_x} \wedge a \in A \setminus X 
\rightarrow (s)_{x} = 0)\big]. \]
$T$ is $\Delta_0(X)$-definable, so it belongs to $\X$, and it is easy to show that it is an infinite tree.
Let $B \in \X^+$ be an infinite branch of $T$. Then $\{x: \tuple{x, y_x} \in X\} = B \cap J(\bba)$. However, $B \cap J(\bba)$ can also be written
as $(B \cap \{0, \ldots, z\}) \cap J(\bba)$ for an arbitrary $z \in M \setminus J(\bba)$, and $B \cap \{0, \ldots, z\}$, being a finite set,
belongs to $\X$. 
\end{proof}

\begin{corollary}\label{cor:system-wkl}
Let $\m = \change{\tuple{M, \X}}$ be a model of $\rcastar$. Let $\bba \in \X$ be a Peano system almost isomorphic to $\tuple{\bbn, S, 0}$. Assume that $J(\bba)$ is a proper cut closed under $\exp$, that $\preccurlyeq$ is a linear ordering on $A$ with least element $c$ and successor function $f$, and that $\oplus, \otimes$ are operations on $A$ which satisfy the usual recursive definitions of addition \emph{resp}.\ multiplication with respect to least element $c$ and successor $f$. Then 
$\tuple{\tuple{\change{A}, \oplus, \otimes, \preccurlyeq, c, f(c)}, \X \cap \P(A)} \models \wklstar$.
\end{corollary}

\begin{proof}
Write $\mathring \bba$ for $\tuple{\change{A}, \oplus, \otimes, \le, c, f(c)}$. By Lemma \ref{lem:system-cut} part (b), we can assume w.l.o.g. that $\bba = \bba(J(\bba))$. Using the fact that $\bba$ is a Peano system, we can prove that for every $x, z \in J(\bba)$:
\begin{align*}
\tuple{x, y_x} \oplus \tuple{z, y_z} & =  \tuple{x+z, y_{x+z}},\\
\tuple{x, y_x} \otimes \tuple{z, y_z} & =  \tuple{x \cdot z, y_{x \cdot z}},\\  
\tuple{x, y_x} \preccurlyeq \tuple{z, y_z} &\textrm{ iff }  x \le z.
\end{align*} 
By the obvious extension of Lemma \ref{lem:system-cut} part (c) to structures with addition, multiplication and ordering, $\tuple{\mathring \bba, \X \cap \P(A)}$ is isomorphic to $\tuple{J(A), \X_{J(A)}}$. Since $J(\bba)$ is proper and closed under $\exp$, this means that $\tuple{\mathring \bba, \X \cap \P(A)} \models \wklstar$. 
\end{proof}

\begin{remark}
It was shown in \change{\cite[Lemma 2.2]{sy:peanocat}} that in $\rca$ a Peano system almost isomorphic to $\bbn$ is actually isomorphic to $\bbn$. In light of Lemma \ref{lem:system-cut}, this is a reflection of the fact that in $\rca$ there are no proper $\Sigma^0_1$-definable cuts. 

Informally speaking, a Peano system which is not almost isomorphic to $\bbn$ is ``too long'', since it contains elements which cannot be obtained by starting at zero and iterating successor finitely many times. On the other hand, a Peano system which is almost isomorphic but not isomorphic to $\bbn$ is ``too short''. The results of this section, together with our Theorem \ref{thm:main}, give precise meaning to the intuitive idea strongly suggested by Table 2 of \change{\cite{sy:peanocat}}, that the problem with characterizing the natural numbers in  $\rcastar$ is ruling out structures that are ``too short'' rather than ``too long''.
\end{remark}

\section{Characterizations: basic case}\label{section:basic-case}

In this section, we prove Theorem \ref{thm:main}.

\begin{theorem1} Let $\psi$ be a second-order sentence in the language with one unary function $f$ 
and one individual constant $c$. If $\wklstar$ proves that $\tuple{\bbn, S, 0} \models \psi$, then over $\rcastar$ the statement 
``for every $\bba$, if $\bba \models \psi$, then there exists a bijection between 
$\bbn$ and $A$'' implies $\rca$.
\end{theorem1}

We use a model-theoretic argument based on the work of Section
\ref{section:almost-iso} and a lemma  about cuts in models of $\mathrm{I}\Delta_0 + \exp + \neg \mathrm{I}\Sigma_1$.

\begin{lemma}\label{lem:cuts}
Let $M \models \mathrm{I}\Delta_0 + \exp + \neg \mathrm{I}\Sigma_1$. There exists a proper $\Sigma_1$-definable cut $J \subseteq M$ closed under $\exp$.
\end{lemma}
\begin{proof}
We need to consider a few cases. 

\emph{Case 1.} $M \models \supexp$. Since $M \not \models \mathrm{I}\Sigma_1$, there exists a $\Sigma_1$ formula $\varphi(x)$, possibly with parameters, which defines a proper subset of $M$ closed under successor. Replacing $\varphi(x)$ by the formula $\hat \varphi(x)$: ``there exists a sequence witnessing that for all $y \le x$, $\varphi(y)$ holds'', we obtain a proper $\Sigma_1$-definable cut $K \subseteq M$. Define:
\[J := \{y: \exists x \in K \, (y < \supexp(x)) \}.\]  
$J$ is a cut closed under $\exp$ because $K$ is a cut, and it is proper because it does not contain $\supexp(b)$ for any $b \notin K$.

The remaining cases all assume that $M \not \models \supexp$. Let $\Log^*(M)$ denote the domain of $\supexp$ in $M$.
By the case assumption and the fact that $M \models \exp$, $\Log^*(M)$ is a proper $\Sigma_1$-definable cut in $M$.
 
\emph{Case 2.} $\Log^*(M)$ is closed under $\exp$. Define $J := \Log^*(M)$.

\emph{Case 3.} $\Log^*(M)$ is closed under addition but not under $\exp$. Let $\Log(\Log^*(M))$ be the subset of $M$ defined as $\{x: \exp(x) \in \Log^*(M)\}$. Since $\Log^*(M)$ is closed under addition, $\Log(\Log^*(M))$ is a cut. Moreover, $\Log(\Log^*(M)) \subsetneq \Log^*(M)$, because $\Log^*(M)$ is not closed under $\exp$. Define:
\[J :=  \{y: \exists x \in \Log(\Log^*(M)) \, (y < \supexp(x)) \}. \]
$J$ is a cut closed under $\exp$ because $\Log(\Log^*(M))$ is a cut, and it is proper because it does not contain 
$\supexp(b)$ for any $b \in \Log^*(M) \setminus \Log(\Log^*(M))$.

\emph{Case 4.} $\Log^*(M)$ is not closed under addition. Let $\frac{1}{2}\Log^*(M)$ be the subset of $M$ defined as 
$\{x: 2x \in \Log^*(M)\}$. Since $\Log^*(M)$ is closed under successor, $\frac{1}{2}\Log^*(M)$ is a cut. Moreover, 
$\frac{1}{2}\Log^*(M) \subsetneq \Log^*(M)$, because $\Log^*(M)$ is not closed under addition. Define:
\[J:= \{y: \exists x \in \tfrac{1}{2}\Log^*(M) \, (y < \supexp(x)) \}. \]
$J$ is a cut closed under $\exp$ because $\frac{1}{2}\Log^*(M)$ is a cut, and it is proper because it does not contain 
$\supexp(b)$ for any $b \in \Log^*(M) \setminus \frac{1}{2}\Log^*(M)$.
\end{proof} 

\begin{remark}
Inspection of the proof reveals immediately that Lemma \ref{lem:cuts} relativizes, in the sense that in a model of 
$\mathrm{I}\Delta_0(X) + \exp + \neg \mathrm{I}\Sigma_1(X)$ there is a $\Sigma_1(X)$-definable proper cut closed under $\exp$.
\end{remark}

\begin{remark}
The method used to prove Lemma \ref{lem:cuts} shows the following result: for any $n \in \omega$, there is a definable cut in $\mathrm{I}\Delta_0 + \exp$, provably closed under $\exp$, which is proper in all models of $\mathrm{I}\Delta_0 + \exp + \neg \mathrm{I}\Sigma_n$.
In contrast, {there is no definable cut in $\mathrm{I}\Delta_0 + \exp$ provably closed under $\supexp$; otherwise, $\mathrm{I}\Delta_0 +  \exp$ would prove its consistency relativized to a definable cut, which would contradict \change{Theorem 2.1} of \cite{pudlak:cuts}.}
\end{remark}



We can now complete the proof of Theorem \ref{thm:main}. 
Assume that $\psi$ is a second-order sentence true of $\tuple{\bbn, S, 0}$ 
provably in $\wklstar$. Let $\m = \tuple{M, \X}$ be a model of $\rcastar + \neg \mathrm{I}\Sigma^0_1$. 
Assume for the sake of contradiction that according to $\m$, the universe of any structure satisfying $\psi$ can be bijectively mapped onto $\bbn$.

Let $J$ be the proper cut in $M$ guaranteed to exist by the relativized version of Lemma \ref{lem:cuts}. 
\change{Note that according to $\m$, there is no bijection between $A_J$ and $\bbn$. Otherwise,
for every $y \in M$ the preimage of $\{0,\ldots,y-1\}$ under the bijection would be a finite subset of $A_J$ of 
cardinality exactly $y$, which would imply $|A_J|_\m = M$. But it is easy to verify that $|A_J|_\m = J$.

From our assumption on $\psi$ it follows that $\m$ believes $\bba(J) \models \neg \psi$.}

By Lemma \ref{lem:system-cut} and its proof, the mapping \change{$\tuple{x,y_x} \mapsto x$ induces an isomorphism between $\tuple{\bba(J), \X \cap \P(A_J)}$ and $\tuple{J, \X_J}$. Since $J$ is closed under addition and multiplication, we can define the operation $\oplus$ on $A_J$ by $\tuple{x,y_x} \oplus \tuple{z,y_z} = \tuple{x+z,y_{x+z}}$, and we can define $\otimes$ and $\preccurlyeq$ analogously. By $\Delta^0_0$ comprehension, $\oplus, \otimes, \preccurlyeq$ are all elements of $\X$. Write $\mathring \bba(J)$ for $\tuple{\bba(J), \oplus, \otimes, \preccurlyeq, \tuple{0,y_0}, \tuple{1,y_1}}$.}

Clearly, $A_J$ with the structure given by $\oplus, \otimes, \preccurlyeq$ satisfies the assumptions of Corollary \ref{cor:system-wkl}, which means that $\tuple{\mathring \bba(J), \X \cap \P(A_J)}$ is a model of $\wklstar$. 
We also claim that $\tuple{\mathring \bba(J), \X \cap \P(A_J)}$
believes $\bbn \models \neg \psi$. This is essentially an immediate consequence of the fact that $\m$ thinks $\bba(J) \models \neg \psi$, since the subsets of $A_J$ are exactly the same in \mbox{$\tuple{\mathring \bba(J), \X \cap \P(A_J)}$} as in $\m$. {There is one minor technical annoyance related to non-unary second-order quantifiers in $\psi$, as the integer pairing function in $\mathring \bba(J)$ does not coincide with that of $M$. The reason this matters is that the language of second-order arithmetic officially contains only unary set variables, so e.g.\ a binary relation is represented by a set of pairs, but a set of $M$-pairs of elements of $A_J$ might not even be a subset of $A_J$.}
Clearly, however, since the graph of the $\mathring \bba(J)$-pairing function is 
$\Delta^0_0(\exp)$-definable in $\m$, a given set of $M$-pairs of elements of $A_J$ belongs to $\X$ exactly if the corresponding  
set of $\mathring \bba$-pairs belongs to $\X \cap \P(A_J)$; and likewise for tuples of greater constant length.

Thus, our claim holds, and we have contradicted the assumption that $\psi$ is true of $\bbn$ provably in $\wklstar$. \hfill $\qed$ (Theorem \ref{thm:main}) 

\medskip

We point out the following corollary of the proof.

\begin{corollary} The following are equivalent over $\rcastar$:
\begin{itemize}
\item[(1)] $\neg \rca$.
\item[(2)] There exists $\m = \change{\tuple{M, \X}}$ satisfying $\wklstar$ such that $|M| \neq |\bbn|$.
\end{itemize}
\end{corollary}

\begin{proof}
$\rca$ proves that all infinite sets have the same cardinality, which gives $(2) \Rightarrow (1)$. To prove $(1) \Rightarrow (2)$,
work in a model of $\rcastar + \neg \rca$ and take the inner model of $\wklstar$ provided by the proof of Theorem \ref{thm:main}.
\end{proof}

\begin{remark}
The type of argument described above can be employed to strengthen Theorem \ref{thm:main} in two ways.

Firstly, it is clear that $\tuple{\bbn, S, 0}$ could be replaced in the statement of Theorem \ref{thm:main} by, for instance,
$\tuple{\bbn, \le, +, \cdot, 0, 1}$. In other words, the extra structure provided by addition and multiplication does not
help in characterizing the natural numbers without $\mathrm{I}\Sigma^0_1$.

Secondly, for any fixed $n \in \omega$, the theories $\rcastar/\wklstar$ appearing in the statement could be extended (both simultaneously) by an axiom expressing the totality of $f_n$, the $n$-th function in the Grzegorczyk-Wainer hierarchy (e.g., the totality of $f_2$ is $\exp$, the totality of $f_3$ is $\supexp$). The proof remains essentially the same, except that the argument used to show Lemma \ref{lem:cuts} now splits into $n+2$ cases instead of four.

By compactness, $\rcastar/\wklstar$ could also be replaced in the statement of the theorem by $\rcastar + \mathrm{PRA}/\wklstar + \mathrm{PRA}$, where PRA is primitive recursive arithmetic.
\end{remark}

\section{Characterizations: exceptions}

In this section, \change{we give a precise statement of Theorem \ref{thm:main2}, 
and prove Theorems \ref{thm:main2} and \ref{thm:negwkl}}.  

\change{\begin{theorem2}
There exists a $\Delta_0$ formula $\Xi(x)$ defining a (polynomial-time recognizable) set 
of $\Sigma^1_1 \wedge \Pi^1_1$ sentences such that $\rcastar$ proves: ``for every $\bba$, $\bba$ is isomorphic to $\tuple{\bbn, S, 0}$ if and only if $\bba \models \xi$ for all $\xi$ such that $\Xi(\xi)$''.
\end{theorem2}}

\change{This is our formulation of ``there exists a set of second-order sentences which provably in $\rcastar$ categorically characterizes the natural numbers''. Note that a characterization by a fixed set of standard sentences is ruled out by Theorem \ref{thm:main} (and a routine compactness argument).} 

\begin{proof}[Proof of Theorem \ref{thm:main2}]
\change{We will abuse notation and write $\Xi$ for the set of sentences defined by the formula $\Xi(x)$.} Let $\Xi$ consist of the sentence $\xi$ from Lemma \ref{lem:almost-cat}
and the sentences 
\[\exists a_0\, \exists a_1\, \ldots \,\exists a_{x-1} \,\exists a_x \left[a_0 = c \wedge a_1 = f(a_0) \wedge \ldots \wedge a_x = f(a_{x-1})\right], \]
for every $x \in \bbn$. (Note that in a nonstandard model of $\rcastar$, the set $\Xi$ will contain sentences of nonstandard length.)

Provably in $\rcastar$, a structure $\bba$ satisfies all sentences in $\Xi$ exactly if it is a Peano system almost isomorphic to $\bbn$ such that for every $x \in \bbn$, $f^x(c)$ exists. Clearly then, $\bbn$ satisfies all sentences in $\Xi$. Conversely, if $\bba$ satisfies all sentences in $\Xi$, then $J(\bba) = \bbn$ and so $\bba$ is isomorphic to $\bbn$.
\end{proof}

\begin{theorem3}
There is a $\Sigma^1_2$ sentence which is a categorical characterization of $\tuple{\bbn, S, 0}$ provably in $\rcastar + \neg \wklrm$.
\end{theorem3}

Before proving the theorem, we verify that the theory it mentions is a $\Pi^1_2$-conservative extension of $\rcastar$. 

\begin{proposition}
The theory $\rcastar + \neg \wklrm$ is a $\Pi^1_2$-conservative extension of $\rcastar$.
\end{proposition}

\begin{proof}
Let $\exists X\, \forall Y \, \varphi(X,Y)$ be a $\Sigma^1_2$ sentence consistent with $\rcastar$. 
Take $\change{\tuple{M, \X}}$ and $A \in \X$ such that  $\change{\tuple{M, \X}} \models \rcastar + \forall Y \, \varphi(A,Y)$.
Let $\Delta_1(A)$-$\mathfrak{Def}$ stand for the collection of the $\Delta_1(A)$-definable subsets
of $M$. $\Delta_1(A)$-$\mathfrak{Def} \subseteq \X$, so obviously 
\change{$\langle M,\Delta_1(A)$-$\mathfrak{Def}\rangle  \models \rcastar + \forall Y\, \varphi(A, Y)$}.
Moreover, by a standard argument, there is a $\Delta_1(A)$-definable infinite binary tree
without a $\Delta_1(A)$-definable branch, so 
\change{$\langle M,\Delta_1(A)$-$\mathfrak{Def}\rangle \models \neg \wklrm$}.
\end{proof}

\begin{proof}[Proof of Theorem \ref{thm:negwkl}]
Work in $\rcastar + \neg \mathrm{WKL}$. The sentence $\psi$, our categorical characterization of $\bbn$, is very much like the the sentence $\xi$ described in the proof of Lemma \ref{lem:almost-cat}, which expressed almost isomorphism to $\bbn$. The one difference is that the $\Sigma^1_1$ conjunct of $\xi$: 
\begin{center}
there exists a discrete linear ordering $\preccurlyeq$ \\ for which $c$ is the least element and $f$ is the successor function,
\end{center}
is strengthened in $\psi$ to the $\Sigma^1_2$ sentence:
\begin{center}
there exist binary operations $\oplus, \otimes$ and a discrete linear ordering $\preccurlyeq$ such that \\ 
$\preccurlyeq$ has $c$ as the least element and $f$ as the successor function, \\
$\oplus$ and $\otimes$ satisfy the usual recursive definition of addition and multiplication, \\
and  such that $\mathrm{I}\Delta_0 + \exp + \neg \mathrm{WKL}$ holds.
\end{center}
$\mathrm{I}\Delta_0 + \exp$ is finitely axiomatizable \change{\cite{gd:mdrp}}, so there is no problem with expressing this as a single sentence. Note that $\psi$ is $\Sigma^1_2$.

Since $\neg \mathrm{WKL}$ holds, the usual $+$, $\cdot$ and ordering on $\bbn$ witness that $\bbn$ satisfies the new $\Sigma^1_2$ conjunct of $\psi$. Of course, $\bbn$ is a Peano system almost isomorphic to $\bbn$, and thus it satisfies $\psi$.

Now let $\bba$ be a structure satisfying $\psi$. Then $\bba$ is a Peano system almost isomorphic to $\bbn$, so we may consider $J(\bba)$. \change{As in the proof of Corollary \ref{cor:system-wkl}, we can show that the canonical isomorphism between $\bba$ and $J(\bba)$ has to map $\oplus, \otimes, \preccurlyeq$ witnessing the $\Sigma^1_2$ conjunct of $\psi$ to the usual $+, \cdot, \le$ restricted to $J$. This guarantees that $J(\bba)$ is closed under $\exp$, because the $\Sigma^1_2$ conjunct of $\psi$ explicitly contains $\mathrm{I}\Delta_0 + \exp$.} Moreover, Corollary \ref{cor:system-wkl} implies that $J(\bba)$ cannot be a proper cut, because otherwise $\bba$ with the additional structure given by $\oplus, \otimes, \preccurlyeq$ would have to satisfy WKL. So, $J(\bba) = \bbn$ and thus $\bba$ is isomorphic to $\bbn$.
\end{proof}

\section{Characterizations: exceptions are exotic}

To conclude the paper, we prove Theorem \ref{thm:from-tanaka} and some corollaries.

\begin{theorem4}
Let $T$ be an  extension of $\rcastar$ conservative for first-order $\forall \Delta_0(\Sigma_1)$ sentences. Let $\eta$ be a second-order sentence consistent with $\wklstar + \supexp$. Then it is not the case that $\eta$ is a categorical characterization of $\tuple{\bbn, S, 0}$ provably in $T$.
\end{theorem4}

\begin{proof}
Let $\m = \change{\tuple{M, \X}}$ be a countable recursively saturated model of $\wklstar + \supexp + \eta$.

Tanaka's self-embedding theorem \cite{tanaka:self} is stated for countable models of $\wkl$. \change{
However, a variant of the theorem is known to hold for $\wklstar$ as well:
\begin{Tanaka}[Wong-Yokoyama, unpublished]
If $\m = \change{\tuple{M, \X}}$ is a countable recursively saturated model of $\wklstar$ and $q\in M$, then there exists a proper cut $I$ in $M$ and an isomorphism $f:{\tuple{M, \X}}\to {\tuple{I, \X_{I}}}$ such that $f(q)=q$.
\end{Tanaka}
\noindent This can be proved by going through the original proof in \cite{tanaka:self} and verifying that all arguments involving $\Sigma^0_1$ induction can be replaced either by $\Delta^0_0(\exp)$ induction plus $\Sigma^0_1$ collection or by saturation arguments\footnote{\change{The one part of Tanaka's proof that does require $\Sigma^{0}_{1}$ induction is making $f$ fix (pointwise) an entire initial segment  rather than just the single element $q$. See \cite{enayat:new-tanaka}.}}. A refined version of the result was recently proved by a different method in \cite{ew:wklstar}. }

Thus, there is a \change{proper} cut $I$ in $M$ such that 
$\change{\tuple{M, \X}}$ and $\change{\tuple{I, \X_I}}$ are isomorphic. In particular, $\change{\tuple{I, \X_I}} \models \eta$.

Let $a \in M \setminus I$. Define the cut $K$ in $M$ to be 
\[\{y: \exists x \in I\, (y < \exp_{a+x}(2))\}.\] 
\change{Since $\exp_{2a}(2) \in M \setminus K$, the cut $K$ is proper and hence $\tuple{K, \X_K} \models \wklstar$. The set $I$ is still a proper cut in $K$, because $a \in K \setminus I$. Furthermore, $I$ is $\Sigma_1$-definable in $K$ by the formula $\exists x\, \exists y\, (y = \exp_{a+x}(2))$.}

$T$ is conservative over $\rcastar$ for first-order $\forall\Delta_0(\Sigma_1)$ sentences, so there is a model 
$\change{\tuple{L, \Y}} \models T$ such that $K \preccurlyeq_{\Delta_0(\Sigma_1)} L$.  We claim that in $\change{\tuple{L, \Y}}$ there is a Peano system $\bba$ satisfying $\eta$ but not isomorphic to $\bbn$. This will imply that $T$ does not prove $\eta$ to be a categorical characterization of $\bbn$. It remains to prove the claim.

We can assume that $\eta$ does not contain a second-order quantifier in the scope of a first-order quantifier. This is because we can always replace first-order quantification by quantification over singleton sets, at the cost of adding some new first-order quantifiers with none of the original quantifiers of $\eta$ in their scope. 

Note that $\change{\tuple{K, \X_K}}$ contains a proper $\Sigma_1$ definable cut, namely $I$, which satisfies $\eta$. Using the universal $\Sigma_1$ formula, we can express this fact by a first-order $\exists \Delta_0(\Sigma_1)$ sentence $\eta^{\mathrm{FO}}$. The sentence $\eta^{\mathrm{FO}}$ says the following:
\begin{center}
there exists a triple ``$\Sigma_1$ formula $\varphi(x,w)$,  parameter $p$, bound $b$'' such that \\
$b$ does not satisfy $\varphi(x,p)$, the set defined by $\varphi(x,p)$ below $b$ is a cut, \\
and this cut satisfies $\eta$.
\end{center}
To state the last part, replace the second-order quantifiers of $\eta$ by quantifiers over subsets of $\{0, \ldots, b-1\}$ (these are bounded first-order quantifiers) and replace the first-order quantifiers by first-order quantifiers relativized to elements below $b$ satisfying $\varphi(x,p)$. By our assumptions about the syntactical form of $\eta$, this ensures that $\eta^{\mathrm{FO}}$ is 
$\exists \Delta_0(\Sigma_1)$.

$L$ is a $\Delta_0(\Sigma_1)$-elementary extension of $K$, so $L$ also satisfies $\eta^{\mathrm{FO}}$. Therefore, $\change{\tuple{L, \Y}}$ also contains a proper $\Sigma_1$-definable cut satisfying $\eta$. \change{The Peano system corresponding to this cut via Lemma \ref{lem:system-cut} also satisfies $\eta$, but it cannot be isomorphic to 
$\bbn$ in $\tuple{L,\Y}$, because its internal cardinality is a proper cut in $L$.} The claim, and the theorem, is thus proved.
\end{proof}

\begin{remark}
The assumption that $\eta$ is consistent with $\wklstar + \supexp$ rather than just $\wklstar$ is only needed to ensure that there is a model of $\rcastar$ with a proper \emph{$\Sigma_1$-definable} cut satisfying $\eta$. The assumption can be replaced by consistency with $\wklstar$ extended by a much weaker first-order statement, but we were not able to make the proof work assuming only consistency with $\wklstar$. 
\end{remark}

One idea used in the proof of Theorem \ref{thm:from-tanaka} seems worth stating as a separate corollary.

\begin{corollary}\label{cor:fo-expressible}
Let $\eta$ be a second-order sentence. The statement ``there exists a Peano system $\bba$ almost isomorphic but not isomorphic to $\tuple{\bbn, S, 0}$ such that \mbox{$\bba \models \eta$}'' is $\Sigma^1_1$ over $\rcastar$.
\end{corollary}
\begin{proof}
By Lemma \ref{lem:system-cut}, a Peano system satisfying $\eta$ and almost isomorphic but not isomorphic to $\bbn$ exists exactly if there is a proper $\Sigma^0_1$-definable cut satisfying $\eta$. This can be expressed by a sentence identical to the first-order sentence $\eta^{\mathrm{FO}}$ from the proof of Theorem \ref{thm:from-tanaka} except for an additional  existential second-order quantifier to account for the possible set parameters in the formula defining the cut.
\end{proof}

Theorem \ref{thm:from-tanaka} also has the consequence that if we restrict our attention to $\Pi^1_1$-conservative extensions of $\rcastar$, then the characterization from Theorem \ref{thm:negwkl} is not only the ``truest possible'', but also the ``simplest possible'' provably categorical characterization of $\bbn$.

\begin{corollary} Let $T$ be a $\Pi^1_1$-conservative extension of $\rcastar$. Assume that the second-order sentence $\eta$ is a categorical characterization of $\tuple{\bbn, S, 0}$ provably in $T$. Then
\begin{itemize}
\item[(a)] $\eta$ is not $\Pi^1_2$, 
\item[(b)] $T$ is not $\Pi^1_2$-axiomatizable.
\end{itemize}
\end{corollary}

\begin{proof} We first prove (b). Assume that $T$ is $\Pi^1_2$-axiomatizable and $\Pi^1_1$-conservative over $\rcastar$.
As observed in \cite{Y2009}, this means that $T + \wklstar$ is $\Pi^1_1$-conservative over $\rcastar$, so $T$ is consistent with $\wklstar + \supexp$. Hence, Theorem \ref{thm:from-tanaka}  implies that there can be no provably categorical characterization of $\bbn$ in $T$.

Turning now to part (a), assume that $\eta$ is $\Pi^1_2$. Since $T$ is $\Pi^1_1$-conservative over $\rcastar$ and proves that $\bbn \models \eta$, then $\rcastar + \eta$ must also be $\Pi^1_1$-conservative over $\rcastar$.  But then, by a similar argument as above, $\eta$ is consistent with $\wklstar + \supexp$, which contradicts Theorem \ref{thm:from-tanaka}.   
\end{proof}


\begin{ack}
We are grateful to Stephen G.\ Simpson for useful \change{remarks and to an anonymous referee for comments which helped improve the presentation}.
\end{ack}

\bibliography{categorical}

\end{document}